\documentclass[11pt]{article}
\usepackage{geometry}                % See geometry.pdf to learn the layout options. There are lots.
\usepackage{amsmath, amsthm, amssymb}
\usepackage{graphicx, color}
\usepackage{pdfsync}
\newtheorem{lemma}{Lemma}

\newtheorem{proposition}[lemma]{Proposition}
\newtheorem{remark}[lemma]{Remark}
\newtheorem{example}[lemma]{Example}
\newtheorem{theorem}[lemma]{Theorem}
\newtheorem{definition}[lemma]{Definition}
\usepackage{graphicx}
\usepackage{amssymb}
\usepackage{epstopdf}
\usepackage{ulem}
\newcommand{\e}{\varepsilon}

\newcommand{\restr}{\ \rule{.4pt}{7pt}\rule{6pt}{.4pt}\ }

\newcommand{\weakstar}{\buildrel * \over \rightharpoonup}

\def\ZZ{\mathbb Z}
\def\NN{{\mathbb N}}
\def\HH{{\mathcal H}}
\def\RR{\mathbb R}

\title{Compactness by coarse-graining\\ in long-range lattice systems}

\author{Andrea Braides\\ \small Dipartimento di Matematica, Universit\`a di Roma Tor Vergata
\\ \small  via della ricerca scientifica 1, 00133 Roma, Italy\and  Margherita Solci \\ \ \small
DADU, Universit\`a di Sassari\\ \small
 piazza Duomo 6, 07041 Alghero (SS), Italy}
\date{}                                           % Activate to display a given date or no date

\begin{document}
\maketitle

\def\L{{\cal L}}

\noindent{\bf Abstract.} We consider energies on a periodic set $\L$ of the form 
$\sum_{i,j\in\L} a^\e_{ij}|u_i-u_j|$, defined on spin functions $u_i\in\{0,1\}$, and we suppose that the typical range of the interactions is $R_\e$ with $R_\e\to +\infty$, i.e., if $|i-j|\le R_\e$ then $a^\e_{ij}\ge c>0$. In a discrete-to-continuum analysis, we prove that the overall behaviour as $\e\to 0$ of such functionals is that of an interfacial energy. The proof is performed using a coarse-graining procedure which associates to scaled functions defined on $\e\L$ with equibounded energy a family of sets with equibounded perimeter. This agrees with the case of equibounded $R_\e$ and can be seen as an extension of coerciveness result for short-range interactions, but is different from that of other long-range interaction energies, whose limit exits the class of surface energies. A computation of the limit energy is performed in the case $\L=\ZZ^d$.



\section{Introduction}
In this paper we give a contribution to the general problem of the asymptotic analysis of systems of lattice interactions of the form 
\begin{equation}\label{1}
\sum_{i,j\in\L} a^\e_{ij}|u_i-u_j|
\end{equation}
where $\L$ is a periodic lattice in $\RR^d$, $\e>0$ is a parameter tending to $0$, and  $a^\e_{ij}$ are non-negative coefficients. 
These functionals depend on (scalar) `spin functions' with $u_i\in\{0,1\}$, somehow related to ferromagnetic energies in the terminology of Statistical Mechanics (where usually $u_i\in\{-1,1\}$). 
%Note that the choice of the cubic lattice 
%as the set of labels in \eqref{1} is made in order to simplify  the exposition but it is not essential for the results.
%Indeed, in this generality all interactions on a Bravais lattice can be reduced to interactions on a square lattice by a change of variables. 

%We will show that for a class of long-range interactions these energies are coercive up to scaling within the class of surface energies. More precisely, we will exhibit a scaling   such that if $u^\e$ are such that  

We investigate coerciveness properties related to such energies in a discrete-to-continuum process, where
the values $u^\e_i$ are identified as the values $u^\e(\e i)$ of a function defined on $\e\L$. In this way 
a continuum limit of $u^\e$ can be defined as a limit of their piecewise-constant interpolations; 
e.g., defined as $u^\e(x)= u^\e_i$ if the point of minimum distance of $\e\L$ from $x$ is $\e i$.
Coerciveness is established by exhibiting scales $s_\e$ such that if $u^\e_i$ are such that 
\begin{equation}\label{2}
\sum_{i,j\in\L} a^\e_{ij}|u^\e_i-u^\e_j|\le s_\e, 
\end{equation}
then the interpolations $u^\e$ are precompact in some topology and their limit points are in general non trivial. This can be expressed by proving that the domain of the $\Gamma$-limit of the scaled energies
\begin{equation}\label{2bis}
{1\over s_\e}\sum_{i,j\in\L} a^\e_{ij}|u^\e_i-u^\e_j| 
\end{equation}
in that topology
is not trivial.

% and the limits of subsequences 
%are characteristic functions of sets of finite perimeter. Moreover, this scaling is optimal in the sense that the scaled energies $\Gamma$-converge to a non-trivial interfacial energy in a discrete-to-continuum process.

%a non-trivial energy of interfacial type can be derived such that the behaviour of minimizers of problems involving energies \eqref{1} (e.g. with some constraint so as to make the problems non-trivial) is described as approximating corresponding minimum problems of the limit energy. This can be expressed in the computation of a non-trivial $\Gamma$-limit.

The simplest case that has been previously treated \cite{CDL,ABC} is nearest-neighbour interactions; i.e, when $a^\e_{ij}$ are strictly positive only when $i,j$ are nearest neighbours (n.n.~for short) in the Delaunay triangulation of $\L$ (e.g.,
$|i-j|=1$ if $\L=\ZZ^d$). In this case choosing $s_\e={\e^{1-d}}$ gives that the scaled energies 
\begin{equation}\label{3}
\sum_{i,j\in\L\ i,j
 \,\rm n.n.} \e^{d-1} a^\e_{ij}|u^\e_i-u^\e_j|
\end{equation}
can be directly seen as a (possibly anisotropic) perimeter of the sets $\{x: u^\e(x)=1\}$ defined through the piecewise-constant interpolation of $u^\e$ from the scaled lattice $\e\L$. Then, the compactness properties of sets of equibounded perimeter ensure the coerciveness in $L^1_{\rm loc}(\RR^d)$ and the limits are characteristic functions. Moreover, the $\Gamma$-limit of the energies can be described
by an energy defined on sets of finite perimeter $A$, which, in the simplest homogenous case, takes the form
\begin{equation}\label{4}
\int_{\partial A} \varphi(\nu)d\HH^{d-1}.
\end{equation}
The same scaling works for finite-range interactions; i.e., when $a^\e_{ij}$ is $0$ if $|i-j|>R$ for some $R$, even though the energies in that case must be interpreted as a non-local perimeter \cite{BP}.
The finiteness of the range of the interactions can be weakened to a decay condition that can be quantified as 
\begin{equation}\label{4bis}
\sup\Bigl\{\sum_{j\in\L\setminus \{i\}}a^\e_{i j}|j-i|: i\in\L,\ \e>0\Bigr\}<+\infty,
\end{equation}
even though the limit energies may have a non-local part if the `tails' of these series are not uniformly negligible \cite{AG}. We note that such analysis is valid beyond pair potentials and generalizes to classes of many-point interactions (see \cite{BK}).

If the decay assumptions \eqref{4bis} do not hold then the `natural' scaling for the energies may be different from the `surface scaling' $\e^{d-1}$, and we might exit the class of interfacial energies. An extreme case is that of `dense graphs', which is better stated in a bounded domain; i.e., when 
considering energies
\begin{equation}\label{5}
\sum_{i,j\in\L\cap{1\over\e} Q} a^\e_{ij}|u_i-u_j|,
\end{equation}
with $Q$ a cube in $\RR^d$, and suppose that $a^\e_{ij}\ge c>0$ for (a positive percentage of) all interactions.
In that case the scaling is $s_\e= {\e}^{-2d}$, and the limit behaviour is described by a more abstract limit functional called a `graphon' energy \cite{Benjamini,Borgs2012,Lovasz2006,Lovasz2012}, which can be viewed as a relaxation of a double integral on $(0,1)$ of the form
\begin{equation}\label{5}
\int_{(0,1)\times(0,1)}W(x,y)|v(x)-v(y)|\,dx\,dy
\end{equation}
defined on $BV((0,1);\{0,1\})$ after a complex and rather abstract relabeling procedure and identification of functions defined on $Q$ with functions defined on $(0,1)$ (see \cite{BCD}).
For sparse graphs (i.e., graphs which are not dense according to the definition above) and interactions not satisfying the decay conditions \eqref{4bis}, the correct scaling, the relative convergence and the form of the $\Gamma$-limit is a complex open problem. In \cite{BCS} an example is given of one-dimensional energies with range $R_\e={1/\sqrt\e}$ such that
a non-trivial $\Gamma$-limit exists for $s_\e={1/\sqrt\e}$ with respect to the $L^\infty$-weak$^*$ convergence, but it is defined on {\it all} functions of bounded variation with values in $[0,1]$ (and not only those with values in $\{0,1\}$). In that example a crucial issue is the topology of the graph of the connections where $a^\e_{ij}\neq0$.

In this paper we consider an intermediate case; i.e., when the decay condition described above does not hold, and $a^\e_{ij}\ge c>0$ when $|i-j|\le R_\e$ with $R_\e>\!>1$ but the topology of the interactions within that range is that of a `dense' graph.
We further make the assumption $\e R_\e<\!<\!1$ so that the discrete-to-continuum process makes sense. We note that this latter condition is not restrictive upon a redefinition of $\e$ in terms of $R_\e$; e.g.~taking $R_\e^{-1/2}$ in the place of $\e$. We keep the dependence of our system on $R_\e$ and $\e$ separate since these parameters may be defined independently in applications. Under these conditions we have 
$$
s_\e={R_\e^{d+1}\over \e^{d-1}},
$$
and with this scaling functions of equi-bounded energy interpolated on the lattice $\e\L$ converge to a characteristic function of a set of finite perimeter. The main argument for obtaining this result is by coarse-graining. Namely, we average the values of $u^\e$ for interaction on cubes with side length of order $R_\e$, so that we can think of those averages as labelled on $\e R_\e\ZZ^d$. We prove first that those labels for which averages are not essentially close to $0$ and $1$ are negligible; hence, we may regard such functions as spin functions defined on a cubic lattice. Then, we show
that the arguments used for nearest-neighbour interactions of \cite{ABC} can be adapted for the interpolated functions
of the averages. Once a limit set of finite perimeter is obtained we can prove the convergence of the interpolations of the original functions to the same set.

As an application of this scaling argument we show that for 
\begin{equation}\label{6}
a^\e_{ij}= a\Bigl({i-j\over R_\e}\Bigr),
\end{equation}
where $a$ is a positive function with $\int a(\xi)|\xi|d\xi$ finite,
the $\Gamma$-limit of the energies 
\begin{equation}\label{7}
F_\e(u)={\e^{d-1}\over R_\e^{d+1}}\sum_{i,j\in\mathbb Z^d} a^\e_{ij}|u_i-u_j|,
\end{equation}
defined on the cubic lattice of $\RR^d$,
is given by an energy as in \eqref{4} with
\begin{equation}\label{8}
\varphi(\nu)=\int_{\RR^d} a(\xi)|\langle \xi,\nu\rangle|d\xi.
\end{equation}
In particular, if $a$ is radially symmetric then \eqref{4} is simply a multiple of the perimeter of $A$.
It is interesting to note that in a sense the case $R_\e\to+\infty$ can be seen as a limit of the case of $R_\e$ finite, for which the $\Gamma$-limit is of the form \eqref{4} with the integrand $\varphi(\nu)$ given by a discretization of the integral in \eqref{8}
(as seen in \cite{Ch,BG, BLB} in a slightly different context). This convergence can be re-obtained using the results in \cite{GS}, where transportation maps are used to transform discrete energies in convolution functionals.

\section{A compactness result}
We denote by $Q_R=[-R/2,R/2)^d$ the (semi-open) coordinate cube centered in $0$ and with side length $R$ in $\RR^d$, by $B_R$ the open ball  centered in $0$ and with side length $R$ in $\RR^d$, and by $e_1,\dots, e_d$ the vectors of the canonical basis of $\mathbb R^d$. Moreover, $\mathcal H^{d-1}$ denotes the $d-1$-dimensional Hausdorff measure and $|\cdot|$ the Lebesgue $d$-dimensional measure.  

Let $\L\subset\RR^d$ be a discrete periodic set. We can suppose without loss of generality that it is periodic in the coordinate directions with period $1$; i.e., $$\L+e_i=\L\hbox{ for all }i\in\{1,\ldots,d\}.$$
The Voronoi cells of $\L$ are defined as
$$
V_i=\{x\in\RR^d: |x-i|<|x-j|\hbox{ for all }j\in\L, \ j\neq i\,\}.
$$
By the periodicity of $\L$ there exists a constant $C_\L>0$ such that
\begin{equation}\label{Voro}
{1\over C_\L}\le |V_i|\le C_\L,\qquad {1\over C_\L}\le\HH^{d-1}(\partial V_i)\le C_\L.
\end{equation}

Each spin function $u\colon\e\mathbb \L\to\{0,1\}$ is identified with its piecewise-constant interpolation, which is the $L^\infty$ function  defined by
$$
u(x)= u(\e i)\ \hbox{ if } x\in\e V_i\hbox{\,; i.e., } |x-\e i|<|x-\e j| \hbox{ for all }j\in\L, \ j\neq i\,.
$$
Note that by \eqref{Voro} the $L^1$ norm of such $u$ is equivalent to
$\e^d\#\{i: u_i\neq 0\}$.

In this section we prove coerciveness properties for energies $E_\e$ defined on
spin functions  $u\colon\e\L\to\{0,1\}$  by 
\begin{equation}\label{fe} 
E_\e(u)=\frac{\e^{2d}}{\eta^{d+1}}\sum_{i,j\in\L,\ i-j\in Q_{\eta/\e}} |u_i-u_j|,
\end{equation}
where we denote $u_i=u(\e i)$, and $\eta=\eta_\e$ are such that
\begin{equation}\label{eta}
\lim_{\e\to 0}\eta_\e=\lim_{\e\to0}{\e\over\eta_\e}=0.
\end{equation}

\begin{lemma}[Compactness]\label{lemma}
Let $u^\e$ be spin functions such that $E_\e(u^\e)$ is equibounded. Then, up to subsequences, the corresponding piecewise-constant interpolations, still denoted by $u^\e$, converge in $L^1_{\rm loc}(\RR^d)$ to $u=\chi_A$, where $A$ is a set of finite perimeter.
\end{lemma}

\begin{proof} The idea of the proof is to subdivide the set of indices $\L$ into disjoint cubes of side-length $\eta/4\e$. The factor $4$ is chosen so that if we consider $i,j$ indices  belonging to two neighbouring cubes with this side-length, respectively, then $i-j\in Q_{\eta/\e}$ so that they interact in energy $E_\e$. In such a way we can associate to each $u^\e$ and each such smaller cube the value $0$ or $1$ of the `majority phase', if such majority phase is sufficiently close to $0$ and $1$, respectively, while we prove that the remaining cubes can be neglected. In this way we will construct coarse-grained functions for which the energy $E_\e$ can be viewed as a standard nearest-neighbour ferromagnetic energy and the compactness then follows by interpreting spin functions as sets of finite perimeter.
\bigskip

For any $k\in\mathbb Z^d$ we set 
$$Q_k^\e=\frac{\eta}{4\e}k+Q_{\frac{\eta}{4\e}}.$$
For $u:\e\L\to\{0,1\}$ we define
$$D(\e,k)(u)=\frac{\big|\#\{i\in Q_k^\e\cap \L: u_i=1\}-\#\{i\in Q_k^\e\cap \L: u_i=0\}\big|}{\# (Q_k^\e\cap \L)}.$$ 
Note that $D(\e,k)(u)$ measures how much the function $u$ is close to its majority phase; more precisely, $D(\e,k)(u)=1$ if $u$ is constant on $Q_k^\e\cap \L$, while $D(\e,k)(u)=0$ if the values of $u$ are equally distributed between $0$ and $1$ in $Q_k^\e\cap \L$.

With fixed $\delta\in(0,1)$, we define 
$$
\mathcal B^\e(u)=\{k\in\mathbb Z^d:D(\e,k)(u)<1-\delta\}.
$$
The  $Q_k^\e$ corresponding to $k\in \mathcal B^\e(u)$ will be considered as the cubes where $u$ is not close to a phase $1$ or $0$. We will first show that such cubes are negligible. 
Indeed, note that thanks to the first inequality in \eqref{Voro} the number of pairs of indices $i,j$ within $Q^\e_k$ are of order $(\eta/\e)^{2d}$ and hence there exists $C_\delta>0$ such that if $k\in \mathcal B^\e(u)$, then the number of `interactions within the cube' $Q_k^\e$ is at least
$C_\delta (\frac{\eta}{\e})^{2d}$; namely,
$$\#\{(i,j): i,j\in Q_k^\e\cap\L, u_i\neq u_j\}
\geq C_\delta \Big(\frac{\eta}{\e}\Big)^{2d}.$$

Hence, if $u^\e$ are as in the hypotheses of the lemma; that is, $F_\e(u^\e)\leq c$, we have
\begin{equation*}
\#\mathcal B^\e(u^\e)\leq\frac{c}{C_\delta}\eta^{1-d}.
\end{equation*}
We can estimate the measures 
\begin{equation}\label{misura-diversi}
\Big|\bigcup_{k\in\mathcal B^\e(u^\e)} \e Q_k^\e\Big|=
\#\mathcal B^\e(u^\e){\eta^d\over 4^d}
\leq\frac{c}{4^dC_\delta}\eta,
\end{equation}
\begin{equation}\label{bordo-diversi}
\HH^{d-1}\Bigl(\partial\bigcup_{k\in\mathcal B^\e(u^\e)} \e Q_k^\e\Big)=
\#\mathcal B^\e(u^\e){2d\,\eta^{d-1}\over 4^{d-1}}
\leq\frac{2cd}{4^{d-1}C_\delta}.
\end{equation}

As for the indices such that $D(\e,k)(u)\geq 1-\delta$, we subdivide them into the sets
\begin{eqnarray*}
\mathcal A^\e_1(u)=\{ k\in\mathbb Z^d: D(\e,k)(u)\geq 1-\delta, \#\{i\in Q_k^\e: u_i=1\}>\#\{i\in Q_k^\e: u_i=0\}\}\\
\mathcal A^\e_0(u)=\{ k\in\mathbb Z^d: D(\e,k)(u)\geq 1-\delta, \#\{i\in Q_k^\e: u_i=1\}<\#\{i\in Q_k^\e: u_i=0\}\}
\end{eqnarray*}
and define
\begin{eqnarray*}
K_j^\e(u)=\bigcup_{k\in \mathcal A^\e_j(u)} \e Q_k^\e\hbox{ for }j=0,1.
\end{eqnarray*}

In order to estimate the measure of the boundary of $K_1^\e(u)$ we estimate the number of cubes $Q_k^\e$ with  
$k\in \mathcal A^\e_1(u)$ which have a side in common with a cube $Q_{k'}^\e$ with 
$k'\in\mathcal A^\e_0(u)$, parameterized on the set
$$
\mathcal A^\e(u):=\{k\in\mathcal A_1^\e(u): k+e_j\in\mathcal A_0^\e(u) \hbox{ for some } j=1,\dots, d\}
$$
To that end, note that if $D(\e,k)(u)\ge 1-\delta$ and $k\in \mathcal A^\e_1(u)$ then 
$$
\#\{i\in Q_k^\e\cap \L: u_i=1\}\ge \Bigl(1-{\delta\over 2}\Bigr)\#\{i\in Q_k^\e\cap \L\},
$$
so that, again recalling the first inequality in \eqref{Voro}, each site $i\in Q_k^\e$ such that $u_i=1$ interacts with $C'_\delta (\frac{\eta}{\e})^{d}$
and conversely for each site $i\in Q_{k'}^\e$ such that $u_i=0$. Hence, the 
interacting pairs $(i,j)\in Q_k^\e\times Q_{k'}^\e$ are at least 
$C'_\delta(\frac{\eta}{\e})^{2d}$.

Hence, 
$$
E_\e(u)\ge \frac{\e^{2d}}{\eta^{d+1}}\#\mathcal A^\e(u)C''_\delta(\frac{\eta}{\e})^{2d}= C''_\delta\#\mathcal A^\e(u)\eta^{d-1}
$$
so that 
$$
\#\mathcal A^\e(u)\le  {1\over C''_\delta}E_\e(u)\eta^{1-d}.
$$

For the functions $u^\e$ we then obtain 
$$
\#\mathcal A^\e(u^\e)\le  \frac{c}{C''_\delta}\eta^{1-d},
$$
so that
\begin{eqnarray*}
\HH^{d-1}(\partial K_1^\e(u^\e))&\le & 2d\biggl(
\HH^{d-1}\Bigl(\bigcup_{k\in\mathcal A^\e(u^\e)} \e Q_k^\e\Bigr)
+\HH^{d-1}\Bigl(\bigcup_{k\in\mathcal B^\e(u^\e)} \e Q_k^\e\Bigr)\biggr)\\
&\le& 2d\Bigl(\#\mathcal A^\e(u^\e) {2d\,\eta^{d-1}\over 4^{d-1}}+\frac{2cd}{4^{d-1}C_\delta}\Bigr)\le C^{'''}_\delta
\end{eqnarray*}
where $C^{'''}_\delta$ is a positive constant depending only on $d, c$ and $\delta$. 
By the compactness of sets of equibounded perimeter this shows that the characteristic functions of the sets $K_1^\e(u^\e)$ are compact in $L^1_{\rm loc}(\RR^d)$. The symmetric argument shows also that $K_0^\e(u^\e)$ are compact in $L^1_{\rm loc}(\RR^d)$. Moreover, if we denote a limit of the sets  $K_j^\e(u^\e)$ by $K_j$ then we have 
\begin{equation}\label{tutto}|\RR^d\setminus (K_0\cup K_1)|=0\end{equation} by
\eqref{misura-diversi}. We highlight the possible dependence of the sets obtained by this procedure on $\delta$ by renaming them 
$K_1^\delta$ and $K_0^\delta$.

Note that if $\delta<\delta'$ then 
$$
K_1^{\delta'}\subset K_1^\delta\hbox{ and } 
K_0^{\delta'}\subset K_0^\delta.
$$
Since in both cases \eqref{tutto} holds, then we must have $K_1^{\delta'}=K_1^\delta$ 
and $K_0^{\delta'}=K_0^\delta$, so that these sets 
are independent of $\delta$ and we may go back to denoting them by $K_1$ 
and $K_0$.

We can now prove the convergence of $u^\e$. Fixed $\delta<1$ as above, we write
$$
u^\e= u^\e\chi_{K_1^\e(u^\e)}+ u^\e\chi_{K_0^\e(u^\e)}+ u^\e\chi_{\RR^d\setminus (K_1^\e(u^\e)\cup K_0^\e(u^\e))}.
$$
By \eqref{misura-diversi} the last term converges to $0$ in $L^1(\RR^d)$
As for the other two terms we localize the convergence by restricting to a cube $Q_R$.
Note that for $k\in \mathcal A^\e_1(u^\e)$ we have
$$
\|u^\e- 1\|_{L^1(\e Q^\e_k)}\le C(1-\delta)\eta^d,
$$
so that 
$$
\|u^\e\chi_{K_1^\e(u^\e)\cap Q_R}- \chi_{K_1^\e(u^\e)\cap Q_R}\|_{L^1(\RR^d)}\le C(1-\delta)R^d 
$$
where $C$ denotes a positive constant not depending on $\delta.$
Analogously for $k\in \mathcal A^\e_0(u^\e)$ we have
$$
\|u^\e\|_{L^1(\e Q^\e_k)}\le C(1-\delta)\eta^d,
$$
and hence 
$$
\|u^\e\chi_{K_0^\e(u^\e)\cap Q_R}\|_{L^1(\RR^d)}\le C(1-\delta)R^d.
$$
We then have, by the local convergence of $K^\e_j(u^\e)$, 
\begin{eqnarray*}&&
\limsup_{\e\to0}\|u^\e \chi_{K_1^\e(u^\e)\cap Q_R}-\chi_{K_1\cap Q_R}\|_{L^1(\RR^d)}
\\
&\le&
\limsup_{\e\to 0}\Bigl(\|u^\e \chi_{K_1^\e(u^\e)\cap Q_R}-\chi_{K^\e_1(u^\e)\cap Q_R}\|_{L^1(\RR^d)}+
\| \chi_{K_1^\e(u^\e)\cap Q_R}-\chi_{K_1\cap Q_R}\|_{L^1(\RR^d)}\Bigr)
\\
&\le&C(1-\delta)R^d,
\end{eqnarray*}
and
$$
\limsup_{\e\to0}\|u^\e \chi_{K_0^\e(u^\e)\cap Q_R}\|_{L^1(\RR^d)}
\le C(1-\delta)R^d,
$$
so that, by the arbitrariness of $\delta$, $u^\e$ converge locally to $\chi_{K_1}$.
\end{proof}

\begin{remark}\rm 
The proof of Lemma \ref{lemma} works exactly in the same way if we suppose that `almost all' pairs of indices of $\L$ within  $Q_{\eta/\e}$ interact; namely, if in place of energy \eqref{fe} we consider
\begin{equation}\label{fe-a} 
E_\e(u)=\frac{\e^{2d}}{\eta^{d+1}}\sum_{i,j\in\L:\ i-j\in Q_{\eta/\e}} a^\e_{ij} |u_i-u_j|,
\end{equation}
with the requirement that there exists $c>0$ such that
\begin{equation}\label{lim-a} 
\lim_{\e\to 0}{\#\{(i,j):i,j\in x+Q_{\eta/\e}: a^\e_{ij}\ge c\}\over \#\{(i,j):i,j\in x+Q_{\eta/\e}\}}=1
\end{equation}
uniformly in $x\in\RR^d$. Condition \eqref{lim-a} is trivially satisfied by energies \eqref{fe} for $c=1$.

Note that condition \eqref{lim-a} cannot be relaxed to `having a proportion' of pairs of indices of $\L$ within  $Q_{\eta/\e}$ interacting, however large this proportion may be below $1$; i.e., it is not sufficient that
\begin{equation}\label{lim-a2} 
\lim_{\e\to 0}{\#\{(i,j):i,j\in x+Q_{\eta/\e}: a^\e_{ij}\ge c\}\over \#\{(i,j):i,j\in x+Q_{\eta/\e}\}}\ge \lambda,
\end{equation}
for any $\lambda<1$. To check this, we may consider the following example: choose $\L=\RR^d$, fix $N\in\NN$, and define
$$
a^\e_{ij}=\begin{cases}
1 &\hbox{ if }i-j\in Q_{\eta/\e}\hbox{ and both }i,j\not\in N\ZZ\\
0 &\hbox{ otherwise.}
\end{cases}
$$
Then \eqref{lim-a2} holds for $\lambda= \bigl(1-{1\over N^d}\bigr)^2$ but, if we define 
$$
u^\e_i=\begin{cases}
1 &\hbox{ if }i\not\in N\ZZ\\
0 &\hbox{ if }i\in N\ZZ,
\end{cases}
$$
then $u^\e$ converge weakly in $L^1_{\rm loc}(\RR^d)$ to the constant $1-{1\over N^d}$.
Since $E_\e(u_\e)=0$ this shows that Lemma \ref{lemma} does not hold.

In this example the subset $N\ZZ^d$ of $\ZZ^d$ can be considered as a `perforation' of the domain and can be treated as such, considering convergence only of the restriction of the functions to $\ZZ^d\setminus N\ZZ^d$ (see Section 3 of \cite{BCPS}). However, the situation can be more complicated if we take
$$
a^\e_{ij}=\begin{cases}
1 &\hbox{ if }i-j\in Q_{\eta/\e}\hbox{ and both }i,j\not\in N\ZZ\\
c_\e &\hbox{ otherwise,}
\end{cases}
$$
that can be regarded as representing a `high-contrast medium', for which the effect of the `perforation' cannot be neglected and for some values of $c_\e$ may give a `double porosity' effect \cite{BCPS}.
\end{remark}

\section{Homogenization of long-range lattice systems}

Let $a:\RR^d\to [0,+\infty)$ be such that $a(\xi)|\xi|$ is Riemann integrable on bounded sets and such that
\begin{equation}\label{a1}
\int_{\RR^d} a(\xi)|\xi|\,d\xi<+\infty,
\end{equation}
and  %$a(0)>0$ so that 
\begin{equation}\label{a2}
 a(\xi)\geq c_0 \hbox{ if } \ |\xi|\leq r_0
\end{equation}
for some $c_0,r_0>0$.

Given $\e,\eta=\eta_\e$ satisfying \eqref{eta} we define the coefficients
$$
a_{ij}^\e=a_{i-j}^\e=a\Big(\frac{\e (i-j)}{\eta}\Big)
$$ 
for $i,j\in\ZZ^d$,
and the energies 
\begin{equation}\label{fe-a} 
F_\e(u)=\frac{\e^{2d}}{\eta^{d+1}}\sum_{i,j\in\mathbb Z^d} a^\e_{i-j}|u_i-u_j|.
\end{equation}

\begin{definition}\label{conv}
A family $\{u^\e\}$ of functions $u^\e:\e\ZZ^d\to\{0,1\}$ {\rm converges} to a set $A\subset\RR^d$ if the piecewise-constant interpolations of $u^\e$ converge to the characteristic function $\chi_A$ in $L^1_{\rm loc}(\RR^d)$ as $\e\to0$.
\end{definition}

By hypothesis \eqref{a2} we may apply Compactness Lemma \ref{lemma}, obtaining that the family $\{F_\e\}$ is coercive with respect to this convergence.

\begin{proposition}
Let $\{u^\e\}$ be such that $\sup_\e F_\e(u^\e)<+\infty$. Then, up to subsequences, there exists a set of finite perimeter $A$ such that $u^\e$ converge to $A$ in the sense of Definition {\rm\ref{conv}}.
\end{proposition}

This coerciveness property justifies the computation of the $\Gamma$-limit of $F_\e$ with respect to the convergence in 
Definition \ref{conv}. We use standard notation in the theory of sets of finite perimeter (see e.g.~\cite{LN98,Maggi}).

\begin{theorem}[Homogenization]\label{theorem}
The functionals defined in \eqref{fe-a} $\Gamma$-converge with respect to the convergence in Definition {\rm\ref{conv}} to the 
functional $F$ defined on sets of finite perimeter by
\begin{equation}\label{f-a} 
F(A)=\int_{\partial^* A} \varphi_a(\nu)d\HH^{d-1},
\end{equation}
where $\partial^* A$ denotes the reduced boundary of $A$, $\nu$ the outer normal to $A$ and $\varphi_a$ is given by
\begin{equation}\label{fi-a}
\varphi_a(\nu)=\int_{\RR^d} a(\xi)|\langle\xi,\nu\rangle|d\xi.
\end{equation}
\end{theorem}

\begin{proof}
In order to better illustrate the proof in the general $d$-dimensional case we first deal with the one-dimensional case, in which we may highlight the coarse-graining procedure without the technical complexities of the higher-order geometry. 
In this case
%, up to replacing $a^\e_{ij}$ with ${1\over 2}(a^\e_{ij}+a^\e_{ji})$ 
we may rewrite the energies as
$$
F_\e(u)=\frac{\e^{2}}{\eta^{2}}\sum_{\xi\in\ZZ}\sum_{i\in\mathbb Z} a^\e_{\xi}|u_{i+\xi}-u_i|.
$$
The relevant 
computation is that of the lower bound for the target $A=[0,+\infty)$. Let $u^\e$ converge to $A$. For each $\xi\in\ZZ\setminus\{0\}$ and $i\in\{1,\ldots,|\xi|\}$ we consider the function $u^\e$ restricted to $\e i+\e \xi\ZZ$. By the $L^1_{\rm loc}$ convergence, we may suppose that each such restriction changes value; i.e., there exists some $k_{i,\xi}\in\ZZ$ such that 
$$
u^\e(\e i+\e k_{i,\xi}\xi)=0 \hbox{ and } u^\e(\e i+\e (k_{i,\xi}+1)\xi)=1.
$$
The set of $\xi$ and $i$ for which this does not hold is negligible for $\e\to0$; the precise proof is directly given for the $d$-dimensional functionals below. For each $\xi$ we then have
$$
\sum_{i=1}^{|\xi|} \sum_{k\in\ZZ}a^\e_{\xi}|u^\e_{i+(k+1)\xi}-u^\e_{i+k\xi}|\ge |\xi| a\Bigl({\e\over\eta}\xi\Bigr),
$$
so that 
\begin{eqnarray*}
\liminf_{\e\to0} F_\e(u^\e)\ge\liminf_{\e\to0}{\e^2\over\eta^2}\sum_{\xi\in\ZZ} |\xi| a\Bigl({\e\over\eta}\xi\Bigr)=\liminf_{\e\to0}\sum_{\xi\in\ZZ}  {\e\over\eta} a\Bigl({\e\over\eta}\xi\Bigr)\Bigl|{\e\over\eta}\xi\Bigr|,
\end{eqnarray*}
the latter being a Riemann sum giving the integral $\displaystyle\int_{\RR}a(\xi)|\xi|\,d\xi$, which is $F(A)$.

\bigskip
We now deal with the $d$-dimensional case. The proof of the lower bound follows the argument above, but is more complex since we must take into account the direction of the interaction vectors $\xi$.

We prove the inequality by applying the blow-up technique (see \cite{fomu} and \cite{BMS}, and for instance \cite{BS,BP,NSS} for the discrete setting). 

We assume that the sequence $\{F_\e(u^{\e})\}$ is equibounded and that $u^\e$ converge in $L^1_{\rm loc}(\RR^d)$ to $u=\chi_A$, where $A$ is a set of finite perimeter. Up to subsequences, we can assume that $\liminf_{\e\to 0} F_\e(u^\e)=\lim_{\e\to 0}F_\e(u^\e)$. 
We define the localized energy on an open set $U$ by
$$
F_\e(u^{\e};U)= \frac{\e^{2d}}{\eta^{d+1}}\sum_{i\in U} \sum_{j\in\mathbb Z^d}a^\e_{i-j}|u^\e_i-u^\e_j|, 
$$
and define the measures $\mu_\e(U)=F_\e(u^{\e};U)$; since the family $\{\mu_\e\}$ is equibounded, we can assume that  $\mu_\e\weakstar \mu$ up to subsequences. 
Now, let $\lambda=\mathcal H^{d-1}\restr \partial^\ast A$; the lower bound inequality 
follows if we show that for $\mathcal H^{d-1}$-a.a.~$x\in\partial^\ast A$ we have 
$$\frac{d\mu}{d \lambda}(x)\geq \varphi_a(\nu),$$ 
where $\frac{d\mu}{d \lambda}$ denotes the Radon-Nikodym derivative of $\mu$ with respect to the Hausdorff $d-1$-dimensional measure $\lambda$. 
By the Besicovitch Derivation Theorem, for $\mathcal H^{d-1}$-a.a.~$x\in\partial^\ast A$ we have that 
$$\frac{d\mu}{d \lambda}(x)=\lim_{\varrho\to 0}\frac{\mu(Q_\varrho^\nu(x))}{\lambda(Q_\varrho^\nu(x))},$$
where $\lambda$ is the measure $\mathcal H^{d-1}\restr \partial^\ast A$, $\nu$ is the normal vector to $\partial^\ast A$ at $x$ and 
$Q_\varrho^\nu(x)$ is a cube centered in $x$ with side length $\varrho$ and a face orthogonal to $\nu$. 
We can fix $x=0$ and denote $Q_\varrho^\nu(0)$ by $Q_\varrho^\nu$. 
Hence, the lower bound follows if we show that 
%$$\frac{d\mu}{d \lambda}(x)\geq \varphi_a(\nu).$$
\begin{equation}\label{lower}
\lim_{\varrho\to 0}\liminf_{\e\to 0}\frac{1}{\varrho^{d-1}}F_\e(u^{\e};Q_\nu^\varrho) \ge \varphi_a(\nu). 
\end{equation}
We may therefore assume that $\varrho=\varrho_\e$ be such that $\frac{\e}{\varrho}\to 0$ and the scaled functions $u^\e(\frac{\e}{\varrho}i)$ interpolated on the lattice $\frac{\e}{\varrho}\ZZ^d$ converge to the characteristic function of the half space $H^\nu=\{x: \langle x,\nu\rangle<0\}$ on $Q_\nu^1$. We define 
$$
A_\e:=\{x\in Q_\nu^1: u^\e(x)\neq \chi_{H^\nu}(x)\},
$$
so that $|A_\e|\to 0$.

If we define
$$I_{\e/\varrho}^\xi=\Bigl\{i\in\mathbb Z^d: \frac{\e}{\varrho} i, \frac{\e}{\varrho} (i+\xi)\in Q_\nu^1\Bigr\}$$
then  
$$
F_\e(u^{\e};Q_\nu^\varrho)\ge \frac{\e^{2d}}{\eta^{d+1}}\sum_{\xi\in\mathbb Z^d}a \Bigl(\frac{\e}{\eta}\xi\Bigr)
\sum_{i\in I_{\e/\varrho}^\xi}
\Bigl|u^\e(\frac{\e}{\varrho} (i+\xi))-u^\e(\frac{\e}{\varrho}i)\Bigr|.$$

\begin{figure}[h!]
\centerline{\includegraphics[width=0.6\textwidth]{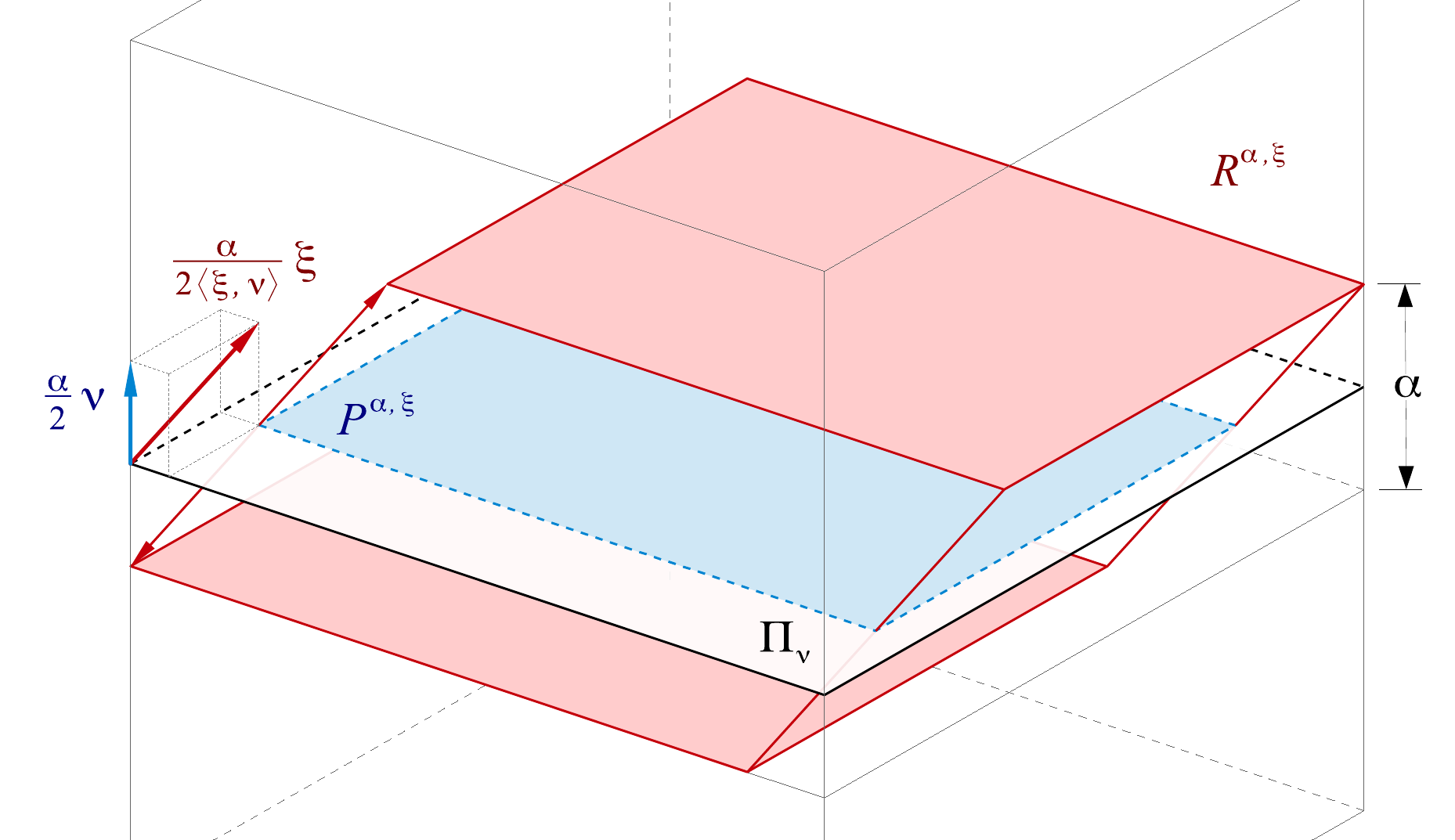}}
\caption{The set $R^{\alpha,\xi}$}
\label{Fig1}
\end{figure} 
We begin by estimating
$$
\sum_{i\in I_{\e/\varrho}^\xi}
\Bigl|u^\e(\frac{\e}{\varrho} (i+\xi))-u^\e(\frac{\e}{\varrho}i)\Bigr|=\#\Bigl\{i\in I_{\e/\varrho}^\xi: 
u^\e(\frac{\e}{\varrho} (i+\xi))\neq u^\e(\frac{\e}{\varrho}i)\Bigr\}.$$
With fixed  $\alpha\in (0,1)$ for each $\xi\in\mathbb Z^d$ satisfying 
\begin{equation}\label{alfa-xi}\Bigl|\langle \frac{\xi}{|\xi|},\nu\rangle\Bigr|\geq  \frac{\alpha}{\sqrt{1+\alpha^2}},
\end{equation}
we define 
$$P^{\alpha,\xi}=\Big\{y\in \Pi_\nu\cap Q_\nu^1: y\pm \frac{\alpha}{2|\langle \xi,\nu \rangle|} \xi\in Q_\nu^1\Big\},$$
which is not empty by \eqref{alfa-xi}, and 
$$R^{\alpha,\xi}=\Big\{x\in Q_\nu^1: x=y+t \xi, \ y\in P^{\alpha,\xi}, \ -\frac{\alpha}{2|\langle \xi,\nu \rangle|}\leq t\leq  \frac{\alpha}{2|\langle \xi,\nu \rangle|}\Big\}$$ 
(see Fig.~\ref{Fig1}). 
Furthermore, we fix $\beta$ with
\begin{equation}\label{beta}\beta> \frac{\alpha}{\sqrt{1+\alpha^2}}.
\end{equation} 
Since we will restrict our arguments to sets $P^{\alpha,\xi}$ and $R^{\alpha,\xi}$ above with $\xi$ satisfying
\begin{equation}\label{beta-xi}\Bigl|\langle \frac{\xi}{|\xi|},\nu\rangle\Bigr|\geq  \beta;
\end{equation}
we omit the dependence of the sets $P^{\alpha,\xi}$ and $R^{\alpha,\xi}$ on $\nu$, since the estimates we will obtain will be independent on $\nu$.

As in the one-dimensional case we consider the functions restricted to the discrete lines $\frac{\e}{\varrho}i+\frac{\e}{\varrho}\xi\mathbb Z$.
The parameter $\alpha$ is introduced so as to estimate the number of sites of such discrete lines inside $Q_\nu^1$.
We then set
$$B^{\alpha,\xi}_{\e/\varrho}=\Bigl\{i\in \mathbb Z^d: \frac{\e}{\varrho}i\in R^{\alpha,\xi} \hbox{ and } 
u^\e \hbox{ is not constant in } \Big(\frac{\e}{\varrho}i+\frac{\e}{\varrho}\xi\mathbb Z\Big)\cap R^{\alpha,\xi}\Bigr\}.$$
Note that if $i\in B^{\alpha,\xi}_{\e/\varrho}$ then $i+k\xi\in B^{\alpha,\xi}_{\e/\varrho}$ for all $k$ with $\frac{\e}{\varrho}(i+k\xi)\in R^{\alpha,\xi}$, so that, if we define the equivalence relation $i\sim i^\prime$ if $i-i^\prime\in \xi\mathbb Z$, 
we may set
$$\widetilde B^{\alpha,\xi}_{\e/\varrho}=B^{\alpha,\xi}_{\e/\varrho}/\sim$$ 
getting
$$
\#\Bigl\{i\in I_{\e/\varrho}^\xi: 
u^\e(\frac{\e}{\varrho} (i+\xi))\neq u^\e(\frac{\e}{\varrho}i)\Bigr\}\geq 
\#\widetilde B^{\alpha,\xi}_{\e/\varrho}.$$

We can estimate the number of `discrete lines' intersecting $R^{\alpha,\xi}$ as
\begin{eqnarray*}
\#\Bigl(\Bigl\{i\in\mathbb Z^d: \frac{\e}{\varrho}i\in R^{\alpha,\xi}\Bigr\}/\sim\Bigr)  &\geq&  
\frac{|R^{\alpha,\xi}|}{\bigl(\frac{\e}{\varrho}\bigr)^d}\frac{1}{
\frac{\alpha}{\frac{\e}{\varrho}|\langle\xi,\nu\rangle|}
}-C_\alpha|\xi|\Bigl(\frac{\e}{\varrho}\Bigr)^{2-d} \\
&\geq&  
\mathcal H^{d-1}(P^{\alpha,\xi})|\langle\xi,\nu\rangle|\Bigl(\frac{\e}{\varrho}\Bigr)^{1-d}-
C_\alpha|\xi|\Bigl(\frac{\e}{\varrho}\Bigr)^{2-d},
%\\
%&\geq&  
%(1-\frac{\alpha}{2} 
%\sqrt{\frac{1}{|\langle \xi/|\xi|,\nu\rangle|^2}-1})
%|\langle\xi,\nu\rangle|\Bigl(\frac{\e}{\varrho}\Bigr)^{1-d}-
%C(\alpha,\xi)\Bigl(\frac{\e}{\varrho}\Bigr)^{2-d}
\end{eqnarray*}
where the last term is an error term accounting for the cubes intersecting the boundary of $ R^{\alpha,\xi}$.

Note that to every element of the complement of  $\tilde B^{\alpha,\xi}_{\e/\varrho}$ 
there correspond at least
$\lfloor \frac{\alpha}{\frac{\e}{\varrho} |\langle \xi, \nu\rangle|}\rfloor$ points in ${\e\over\varrho}\ZZ^d\cap A_\e\}$, so that
for $\e$ sufficiently small we get
\begin{eqnarray*}
&&\#\Bigl(\Bigl\{i\in \mathbb Z^d: \frac{\e}{\varrho}i\in R^{\alpha,\xi} \hbox{ and } 
u^\e \hbox{ constant in } \big(\frac{\e}{\varrho}i+\frac{\e}{\varrho}\xi\mathbb Z\big)\cap R^{\alpha,\xi}\Bigr\} 
/\sim \Bigr)\\
&&\hspace{1cm}\leq \frac{|A_\e|}{\bigl(\frac{\e}{\varrho}\bigr)^d} \frac{\frac{\e}{\varrho}|\langle \xi, \nu\rangle|}{\alpha} + C'_\alpha|\xi|\Bigl(\frac{\e}{\varrho}\Bigr)^{2-d}\\
&&\hspace{1cm}= \frac{1}{\alpha}|A_\e| \Bigl(\frac{\e}{\varrho}\Bigr)^{1-d}
|\langle \xi, \nu\rangle| + C'_\alpha|\xi|\Bigl(\frac{\e}{\varrho}\Bigr)^{2-d},
\end{eqnarray*}
with $C'_\alpha$ again a positive constant accounting for boundary cubes,
and hence \begin{equation}\label{stima}
\#\widetilde B^{\alpha,\xi}_{\e/\varrho}\geq
\mathcal H^{d-1}(P^{\alpha,\xi})|\langle\xi,\nu\rangle|\Bigl(\frac{\e}{\varrho}\Bigr)^{1-d}-\frac{1}{\alpha}|A_\e| \Bigl(\frac{\e}{\varrho}\Bigr)^{1-d}
|\langle \xi, \nu\rangle|
-
(C_\alpha+C'_\alpha)|\xi|\Bigl(\frac{\e}{\varrho}\Bigr)^{2-d}.
\end{equation}
By \eqref{beta-xi} we can estimate
\begin{equation}\label{misuraP}
\mathcal H^{d-1}(P^{\alpha,\xi})\geq \Bigl(1-\frac{\alpha}{2|\langle\xi,\nu\rangle|}|\xi-\langle\xi,\nu\rangle\nu|\Bigr)^{d-1}
\geq \Big(1-\frac{\alpha}{2}\sqrt{\frac{1}{\beta^2}-1}\Big)^{d-1}
\end{equation}
(see also Fig.~\ref{Fig1}), and hence, upon fixing $R>0$ and introducing the set 
$$\Xi_\e^\nu(R,\beta)=
\Bigl\{\xi\in \mathbb Z^d: |\xi|\leq \frac{\eta}{\e}R, \Bigl|\langle \frac{\xi}{|\xi|},\nu \rangle\Bigr|\geq \beta\Bigr\},$$ 
by \eqref{stima} and \eqref{misuraP} we have  
\begin{eqnarray*}
\frac{1}{\varrho^{d-1}}F_\e(u^{\e};Q_\nu^\varrho) &\geq & 
\frac{1}{\varrho^{d-1}}\frac{\e^{2d}}{\eta^{d+1}}\sum_{\xi\in\Xi_\e^\nu(R,\beta)}a \Bigl(\frac{\e}{\eta}\xi\Bigr)
\mathcal H^{d-1}(P^{\alpha,\xi})|\langle\xi,\nu\rangle|\Bigl(\frac{\e}{\varrho}\Bigr)^{1-d}\\
&&-\frac{1}{\varrho^{d-1}}\frac{\e^{2d}}{\eta^{d+1}}\sum_{\xi\in\Xi_\e^\nu(R,\beta)}a \Bigl(\frac{\e}{\eta}\xi\Bigr)
\frac{1}{\alpha}|A_\e| \Bigl(\frac{\e}{\varrho}\Bigr)^{1-d}
|\langle \xi, \nu\rangle|\\
&&-
(C_\alpha+C'_\alpha)\Bigl(\frac{\e}{\varrho}\Bigr)^{2-d}
\frac{1}{\varrho^{d-1}}\frac{\e^{2d}}{\eta^{d+1}}
\sum_{\xi\in\Xi_\e^\nu(R,\beta)}a \Bigl(\frac{\e}{\eta}\xi\Bigr)|\xi|\\
&\geq& 
\Big(1-\frac{\alpha}{2}\sqrt{\frac{1}{\beta^2}-1}\Big)^{d-1}
\sum_{\xi\in\Xi_\e^\nu(R,\beta)}\Bigl(\frac{\e}{\eta}\Bigr)^{d}a \Bigl(\frac{\e}{\eta}\xi\Bigr)
|\langle \frac{\e}{\eta}\xi,\nu\rangle|\\
&&-\frac{1}{\alpha}|A_\e| \sum_{\xi\in\Xi_\e^\nu(R,\beta)}\Bigl(\frac{\e}{\eta}\Bigr)^{d}a \Bigl(\frac{\e}{\eta}\xi\Bigr)
|\langle \frac{\e}{\eta}\xi,\nu\rangle|\\
&&-
(C_\alpha+C'_\alpha)
\frac{\e}{\varrho}
\sum_{\xi\in\Xi_\e^\nu(R,\beta)}\Bigl(\frac{\e}{\eta}\Bigr)^{d}a \Bigl(\frac{\e}{\eta}\xi\Bigr)\Bigl|\frac{\e}{\eta}\xi\Bigr|.
\end{eqnarray*}

Since $|A_\e|\to 0$ and 
\begin{eqnarray*}
 \lim_{\e\to 0}\sum_{\xi\in\Xi_\e^\nu(R,\beta)}\Bigl(\frac{\e}{\eta}\Bigr)^{d}a \Bigl(\frac{\e}{\eta}\xi\Bigr)
\Bigl|\langle \frac{\e}{\eta}\xi,\nu\rangle\Bigr|
&=&\int_{\{|\xi|\leq R, |\langle \xi/|\xi|, \nu\rangle|\ge \beta\}}a(\xi)|\langle \xi,\nu\rangle|\, d\xi,\\
 \lim_{\e\to 0}\sum_{\xi\in\Xi_\e^\nu(R,\beta)}\Bigl(\frac{\e}{\eta}\Bigr)^{d}a \Bigl(\frac{\e}{\eta}\xi\Bigr)
\Bigl|\frac{\e}{\eta}\xi\Bigr|
&=&\int_{\{|\xi|\leq R, |\langle \xi/|\xi|, \nu\rangle|\ge \beta\}}a(\xi)| \xi|\, d\xi,
\end{eqnarray*}
we get\begin{eqnarray*}\liminf_{\e\to0}
\frac{1}{\varrho^{d-1}}F_\e(u^{\e};Q_\nu^\varrho) &\geq & 
\Big(1-\frac{\alpha}{2}\sqrt{\frac{1}{\beta^2}-1}\Big)^{d-1}
\int_{\{|\xi|\leq R, |\langle \xi/|\xi|, \nu\rangle|\ge\beta\}}a(\xi)|\langle\xi,\nu\rangle|\, d\xi .
\end{eqnarray*}
Note that by \eqref{beta} we may let first  $\alpha\to 0$ and then $\beta\to 0$.
We eventually obtain
\begin{eqnarray*}
\liminf_{\e\to0}\frac{1}{\varrho^{d-1}}F_\e(u^{\e};Q_\nu^\varrho) \geq 
\int_{\{|\xi|\le R\}}a(\xi)|\langle\xi,\nu\rangle|\, d\xi,
\end{eqnarray*}
which, by the arbitrariness of $R$, gives \eqref{lower}.

\begin{figure}[h!]
\centerline{\includegraphics[width=0.8\textwidth]{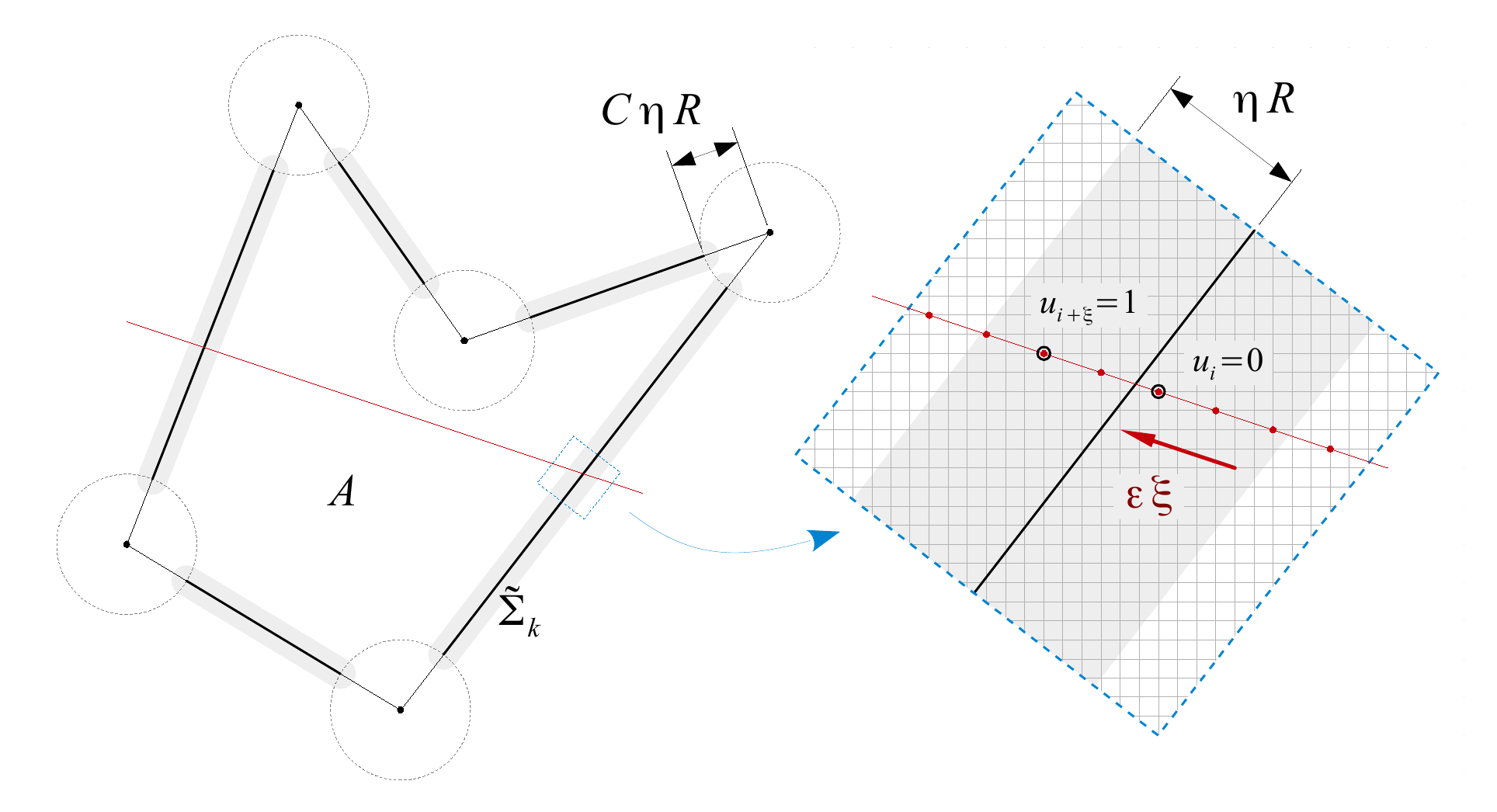}}
\caption{Upper-bound construction}
\label{Fig2}
\end{figure} 
\bigskip
The upper bound is obtained by a density argument (see \cite{GCB} Section 1.7). Hence, it suffices to treat the case of $A$ polyhedral. In this case it suffices to take (the interpolations) $u^\e_i=\chi_A(\e i)$ for $i\in\ZZ^d$. 
Indeed, 
we write $\partial A$ as a union of $N$ $d-1$-dimensional polytopes $\Sigma_k$ and we denote by $\nu_k$ the outer normal to $\Sigma_k$ and by $K$ the $d-2$-dimensional skeleton of $A$. 

We note that there exists a constant $C$ depending only on $A$ such that, for any $\eta,R>0$,
after removing the closed neighborhood
$K+\overline B_{C \eta R}$ from $\partial A$,  
%the neighborhood $\partial A+B_{\eta R}$, 
we obtain a disjoint collection $\widetilde \Sigma_1,\dots, \widetilde \Sigma_N$ 
with $\widetilde \Sigma_k\subset \Sigma_k$ such that 
$$\Big(\widetilde \Sigma_k+B_{\eta R} \Big)\cap \Big(\widetilde \Sigma_{k'}+B_{\eta R}\Big)=\emptyset 
\ \ \hbox{ for any } k\neq k'$$
%the disjoint collection  
%$$\bigcup_{k=1}^N\Big(\widetilde \Sigma_k+B_{\eta R}\Big), \ \  \widetilde \Sigma_k=\Sigma_k \setminus 
%\Big(K+\overline B_{C \eta R}\Big)$$
(see Fig.~\ref{Fig2}). 
Hence, for any $\xi\in\ZZ^d$ with $|\e\xi|\le \eta R$, $k\in\{1,\ldots,N\}$, and 
$j\in\ZZ^d$ such that the line $\e j+\e\xi\RR$ intersects $\widetilde \Sigma_k$, 
%at a point of distance larger than $C\eta R$ from the relative boundary of $\Sigma_k$, then 
the values $u^\e_i$ change only once on the points of the discrete lines $\e j+\e \xi\ZZ$ which lie in a $\eta R$ neighbourhood of $%{1\over\e}
\widetilde\Sigma_k$. We note that for lines intersecting $\Sigma_k$ at a point of distance not larger than $C\eta R$ from $K$, such changes of value are at most $N$; then, repeating the counting argument used in the lower bound, we obtain
\begin{eqnarray*}
&&\limsup_{\e\to 0}\frac{\e^{2d}}{\eta^{d+1}}\sum_{\substack{\xi\in\ZZ^d\\ |\e\xi|\leq\eta R}} \!\!a \Bigl(\frac{\e}{\eta}\xi\Bigr)\sum_{i\in\ZZ^d}|u^\e_{i+\xi}-u^\e_i|\\
&&\hspace{1cm}\leq\limsup_{\e\to 0}\Biggl(\sum_{k=1}^N \mathcal H^{d-1}(\Sigma_k)\frac{\e^{d}}{\eta^{d}}\sum_{\substack{\xi\in\ZZ^d\\ |\e\xi|\leq\eta R}} \!\!a \Bigl(\frac{\e}{\eta}\xi\Bigr)
\Bigr|\langle \frac{\e}{\eta}\xi, \nu_k\rangle\Bigl| +O(\eta R)\Biggr)\\
&&\hspace{1cm}\leq\limsup_{\e\to 0}\sum_{k=1}^N  \mathcal H^{d-1}(\Sigma_k)\int_{ \{|\xi|\leq R\}} a(\xi)
|\langle \xi, \nu_k\rangle|\, d\xi \\
&&\hspace{1cm}\leq \int_{\partial A}\varphi_a(\nu)d\HH^{d-1}.
\end{eqnarray*}
Since also for $|\e\xi|\ge \eta R$  the changes of value of $u^\e_i$ are at most $N$, we then get
$$\limsup_{\e\to 0}\frac{\e^{2d}}{\eta^{d+1}}\sum_{\substack{\xi\in\ZZ^d\\ |\e\xi|>\eta R}} \!\!a \Bigl(\frac{\e}{\eta}\xi\Bigr)\sum_{i\in\ZZ^d}|u^\e_{i+\xi}-u^\e_i|\leq 
N\HH^{d-1}(\partial A)\int_{\{|\xi|>R\}}a(\xi)|\xi|d\xi.
$$
Since this term vanishes as $R\to+\infty$ the upper bound follows.
\end{proof}

\begin{example}\label{example}\rm
If $a$ is radially symmetric, then we have 
\begin{equation}\label{fe-symm} 
F(A)= \sigma \HH^{d-1}(\partial^*A),
\end{equation}
where $\sigma$ is given by
\begin{equation}\label{fe-symm-s} 
\sigma=\int_{\RR^d}a(\xi)|\xi_1|d\xi.
\end{equation}

In particular, we may take $a= \chi_{B_1}$ the characteristic function of the unit ball in $\RR^d$. In this case the limit of 
\begin{equation}\label{fe-ball} 
F_\e(u)=\frac{\e^{2d}}{\eta^{d+1}}\sum_{i,j\in\mathbb Z^d\ |i-j|<\eta/\e} |u_i-u_j|.
\end{equation}
is given by 
\begin{equation}\label{fe-ball-s} 
\sigma=\int_{B_1}|\xi_1|d\xi.
\end{equation}
\end{example}

\begin{remark}[local version]\label{remark}\rm
If $\Omega\subset\RR^d$ is an open set with Lipschitz boundary we may define
\begin{equation}\label{fe-a-omega} 
F_\e(u)=\frac{\e^{2d}}{\eta^{d+1}}\sum_{i,j\in\mathbb Z^d\cap{1\over\e}\Omega} a^\e_{i-j}|u_i-u_j|.
\end{equation}
Then the $\Gamma$-limit is 
\begin{equation}\label{f-a-omega} 
F(A)=\int_{\Omega\cap\partial^* A} \varphi_a(\nu)d\HH^{d-1},
\end{equation}
with minor modifications in the proof.
\end{remark}

\bigskip

\noindent{\bf Acknowledgments.}
Andrea Braides acknowledges the MIUR Excellence Department Project awarded to the Department of Mathematics, University of Rome Tor Vergata, CUP E83C18000100006. Margherita Solci acknowledges 
the project ``Fondo di Ateneo per la ricerca 2019", funded by the University of Sassari. We thank the anonymous referee of \cite{BCS}, who drew our attention to the problem in this paper.

\end{document}